\documentclass[12pt,leqno,oneside]{amsart}
\usepackage{mathrsfs,dsfont}
\usepackage{amsmath, amstext, amsthm, amssymb, bbm, comment}
\usepackage{charter}
\usepackage{typearea}
\usepackage{hyperref}
\usepackage{color}
\usepackage{nccmath}
\usepackage{float}
\usepackage{tikz}
\usetikzlibrary{patterns}
\usepackage{graphicx}
\usepackage{hyperref}

\usepackage[width=6.4in,height=8.5in]
{geometry}

\def\bt{\begin{thm}}
	\def\et{\end{thm}}
\def\bl{\begin{lem}}
	\def\el{\end{lem}}
\def\bd{\begin{defn}}
	\def\ed{\end{defn}}
\def\bc{\begin{cor}}
	\def\ec{\end{cor}}
\def\bp{\begin{proof}}
	\def\ep{\end{proof}}
\def\br{\begin{rem}}
	\def\er{\end{rem}}

\newtheorem{thm}{Theorem}[section]

\newtheorem{lem}[thm]{Lemma}
\newtheorem{defn}[thm]{Definition}
\newtheorem{example}[thm]{Example}
\newtheorem{rem}[thm]{Remark}
\newtheorem{cor}[thm]{Corollary}

\newtheorem{proposition}[thm]{Proposition}

\numberwithin{equation}{section}

\title{Zero Distribution of Random Bernoulli Polynomial Mappings }

\author{Turgay Bayraktar \& \c{C}\.i\u{g}dem \c{C}el\.ik} 
\thanks{T.B. and \c{C}.\c{C}. are partially supported by T\"{U}B\.{I}TAK grant ARDEB-1001/119F184}
\address{Faculty of Engineering and Natural Sciences, Sabanc{\i} University, \.{I}stanbul, Turkey}
\email{tbayraktar@sabanciuniv.edu}
\email{cigdemcelik@sabanciuniv.edu}

\begin{document}
	
	\begin{abstract}
	In this note, we study asymptotic zero distribution of multivariable full system of random polynomials with independent Bernoulli coefficients. We prove that with overwhelming probability their simultaneous zeros  are discrete and the associated normalized empirical measure of zeros asymptotic to the Haar measure on the unit torus.
	\end{abstract}
	
	\maketitle
\section{Introduction}
	A random Kac polynomial on the complex plane is of the form 
\begin{equation}\label{rp}	
	f_d(z)=\sum_{j=0}^{d}a_jz^j
\end{equation}	
	where the coefficients $a_j$ are independent copies of the (real or complex) standard Gaussian. A classical result due to Kac, Hammersley and Shepp \& Vanderbei \cite{Kac, Ham, SV} asserts that almost surely the normalized  empirical measure of zeros $\delta_{Z(f_d)}:=\frac1d\sum_{f_d(\zeta)=0}\delta_{\zeta}$, converges to normalized arc length measure on $S^1:=\{|z|=1\}$ as $d\to\infty$. Asymptotic zero distribution of Kac polynomials with independent identically distributed (i.i.d.) discrete random coefficients have also been studied extensively (see eg. \cite{LO,ET}). More recently, Ibragimov and Zaporozhets \cite{IZ11} proved that the empirical measure of zeros $\delta_{Z(f_d)}$ almost surely converges to the the normalized arc length measure if and only if the moment condition $\mathbb{E}[\log(1+|a_i|)]<\infty$ holds. This property can be considered as a global universality property of the zeros of random polynomials (see also \cite{TV} for a local version).  
	
	Building upon the work of Shiffman and Zelditch \cite{SZ}, equilibrium distribution of random systems of polynomials with Gaussian coefficients was obtained by Bloom \& Shiffman \cite{BS} and Shiffman \cite{S}. More recently, these results were generalized for i.i.d. random coefficients with bounded density \cite{TBI,bay}. We refer the reader to the survey \cite{BCHM} and references therein for the state of the art. On the other hand, asymptotic zero distribution of random polynomial mappings with discrete random coefficients remained open (cf. \cite{B,BD,BBL}). In this note, we study asymptotic zero distribution of multivariable full system of random polynomials with independent Bernoulli coefficients. 
	
\subsection{Statement of the results}
A \emph{ random Bernoulli polynomial} is of the form
	\begin{equation*}\label{eq:pol}
		f_{d,i}(\boldsymbol{x})=\sum_{|\textbf{$J$}|\leq d}\alpha_{i,\textbf{$J$}}\boldsymbol{x}^\textbf{$J$} \in \mathbb{C}\left[x_1, \ldots, x_{n}\right] 
	\end{equation*}
	
	\noindent where $\boldsymbol{x}^{J} = x_{1}^{j_1}\ldots x_{n}^{j_n}$ and $\alpha_{i,J}$ are $\pm1$ Bernoulli random variables for $i=1,\ldots,n$. Throughout this work, we consider systems $(f_{d,1},\ldots,f_{d,n})$ of random Bernoulli polynomials with independent coefficients. We write  $\boldsymbol{f}_d=\left( f_{d,1},\ldots,f_{d,n} \right)$ for short. We denote the collection of all systems of polynomials in $n$ variables and of degree $d$   by $Poly_{n,d}$ that is endowed with the product probability measure $Prob_d$. 
	\begin{thm}{\label{th:main}} Let $\boldsymbol{f}_d=(f_{d,1},\ldots,f_{d,n})$ be a system of random polynomials with independent $\pm1$ valued Bernoulli coefficients. Then there exists a dimensional constant $K=K(n)>0$ and an exceptional set $\mathcal{E}_{n,d}\subset Poly_{n,d}$ such that $Prob_d(\mathcal{E}_{n,d})\leq K/d$ and for all $\boldsymbol{f}_d\in	Poly_{n,d}\setminus\mathcal{E}_{n,d}$ the simultaneous zeros $Z(\boldsymbol{f}_d)$ of the system $\boldsymbol{f}_d$ are isolated with $\#Z(\boldsymbol{f}_d)=d^n.$
	\end{thm}	
For a system $\boldsymbol{f}_d\in Poly_{n,d}$, if the simultaneous zeros $Z(\boldsymbol{f}_d)$ are isolated we denote the corresponding normalized empirical measure by $\delta_{Z(\boldsymbol{f}_d)}$. That is $\delta_{Z(\boldsymbol{f}_d)}$ is a probability measure supported on the isolated zeros with equal weight on each zero. We also let  $\nu_{\text{Haar}}$ denote the Haar measure on $(S^1)^n$ of total mass 1.	
As an application of Theorem \ref{th:main} together with a deterministic equidistribution result \cite[Theorem 1.7]{DGS}, we obtain asymptotic zero distribution of random Bernoulli polynomial mappings:	
	\begin{cor}\label{cor:eq} Let $\boldsymbol{f}_d=(f_{d,1},\ldots,f_{d,n})$ be system of random polynomials with independent $\pm1$ valued Bernoulli coefficients and $\mathcal{E}_{n,d}\subset Poly_{n,d}$ be as in Theorem \ref{th:main}. Then for each sequence $\boldsymbol{f}_d\in	Poly_{n,d}\setminus\mathcal{E}_{n,d}$ we have  
			\begin{equation*}
			\lim_{d\rightarrow\infty}\delta_{Z(\boldsymbol{f}_d)}= \nu_{\text{Haar}}.
		\end{equation*} in the weak topology.
	In particular, $\delta_{Z(\boldsymbol{f}_d)}\to\nu_{Haar}$ in probability $Prob_d$ as $d\to\infty$.
\end{cor}		
		
Finally, we consider the measure valued random variables 	
\begin{equation}\label{zeromeasure}
\widetilde{Z}(\boldsymbol{f}_d):=
\begin{cases}
	\sum_{\xi_i\in Z(\boldsymbol{f}_d)} \delta(\xi_i) &  \text{for}\  \boldsymbol{f}_d\in Poly_{n,d}\setminus\mathcal{E}_{n,d} \\
	
	0 & \text{otherwise} \\
\end{cases}
\end{equation}
and define \textit{the expected zero measure} by 
\begin{equation}
	\left\langle \mathbb{E}[\widetilde{Z}(\boldsymbol{f}_d)],\varphi\right\rangle = \int_{Poly_{n,d}\setminus\mathcal{E}_{n,d}} \sum_{\xi_i\in Z(\boldsymbol{f}_d)}\varphi(\xi_i)\ dProb_d(\boldsymbol{f}_d)
\end{equation}
where $\varphi$ is a continuous function with compact support in $\mathbb{C}^n$ and $\mathcal{E}_{n,d}$ denote the exceptional set given by Theorem \ref{th:main}.
\begin{thm}\label{th:exp}
	Let $\boldsymbol{f}_d=(f_{d,1},\ldots,f_{d,n})$ be a system of random polynomials with independent $\pm1$ valued Bernoulli coefficients. Then
	\begin{equation*}
	 \lim_{d\to\infty}d^{-n}\mathbb{E}[\widetilde{Z}(\boldsymbol{f}_d)]=\nu_{\text{Haar}}
	\end{equation*} in the weak topology.

\end{thm}


The outline of this work as follows. In \S 2, we review some basic properties of resultants. In particular, we recall multi-polynomial resultant and sparse resultant for polynomial systems \cite{GKZ,Cox} as well as directional resultant \cite{DS}.  In \S 3, we prove the main result Theorem \ref{th:main}. Finally, in \S 4 we prove Theorem \ref{th:exp}.
		
	
\section{Preliminaries}
In this section, we review some basic results in algebraic geometry and discrepancy theory related to our results. More precisely, we discuss the multi-homogenous (classical) resultant and the sparse eliminant as well as the relation of these two notions. For a detailed account of the subject and proofs we refer the reader to \cite{GKZ,Cox}. We also discuss the sparse resultant introduced by  D'Andrea and Sombra, and corresponding directional sparse resultants \cite{DGS,DS}.
	\subsection{Lattice points, polytopes}
	For a nonempty subset $P\subset\mathbb{R}^n$, we denote its convex hull in $\mathbb{R}^n$ by $conv(P)$. For two nonempty convex sets $Q_1,Q_2$, their Minkowski sum is defined as 
	$$Q_1+Q_2:=\{q_1+q_2: q_1\in Q_1, q_2\in Q_2\}$$
	 and for $\lambda\in\mathbb{R}$, the scaled polytope is of the form
	 $$\lambda Q:=\{\lambda q: q\in Q\}.$$	
	 It is well known that $Vol_n(d_1Q_1+\ldots+d_nQ_n)$ is a homogenous polynomial of degree $n$ in the variables $d_1,\ldots,d_n\in\mathbb{Z}_{+}$ where $Vol_n$ denotes the normalized volume of the subsets in $\mathbb{R}^n$ with respect to the Lebesgue measure. The coefficient of the monomial $d_1\ldots d_n$ is called the \textit{mixed volume} of $Q_1,\ldots,Q_n$ and denoted by $MV(Q_1,\ldots,Q_n)$. One can use the polarization formula to compute the mixed volume of the convex sets $Q_1,\ldots,Q_n$. Namely, 
	 $$MV_n(Q_1,\ldots,Q_n)=\sum_{k=1}^{n}\sum_{1\leq j_1\leq\ldots\leq j_k\leq n}(-1)^{n-k}Vol_n(Q_{j_1}+\ldots+Q_{j_k}).$$
	In particular, if $Q=Q_1=\ldots=Q_n$ then 
	
	$$ MV_n(Q):=MV_n(Q,\ldots,Q)=n!Vol_n(Q).$$

	For a convex set $Q\subset\mathbb{R}^n$ its \textit{support function} $s_{Q}:\mathbb{R}^n\to\mathbb{R}$ is defined by 
	\begin{equation}\label{support function}
		s_{Q}(\boldsymbol{v}):=\inf_{\boldsymbol{q}\in Q}\left\langle \boldsymbol{q},\boldsymbol{v}\right\rangle
	\end{equation} 
	where $\left\langle \cdot,\cdot\right\rangle $ represents the Euclidean inner product of $\mathbb{R}^n$. Given a vector $\boldsymbol{v}\in\Bbb{R}^n$ the equation $$ \left\langle \boldsymbol{q},\boldsymbol{v}\right\rangle = s_{Q}(\boldsymbol{v})$$
defines \textit{supporting hyperplane} of $Q$ and $\boldsymbol{v}$ is called an \textit{inward pointing normal}. The intersection of $Q$ with the supporting hyperplane in the direction $\boldsymbol{v}\in\mathbb{R}^n$ is denoted by 

\begin{equation}\label{supporting hyperplane}Q^{\boldsymbol{v}}:=\{\boldsymbol{q}\in Q: \left\langle \boldsymbol{q},\boldsymbol{v}\right\rangle = s_{Q}(\boldsymbol{v}) \}.\end{equation}
The set $Q^{\boldsymbol{v}}$ is called the \textit{face} of $Q$ determined by $\boldsymbol{v}$. If $Q^{\boldsymbol{v}}$ has codimension 1, it is called a \textit{facet} of $Q$.

\subsection{Resultant of polynomial systems}
	\subsubsection{Multipolynomial Resultant}
	We consider homogenous polynomials of degree $d_i\geq 0$ of the form
	$$\textbf{F}_i(t_0,\ldots,t_n)=\sum_{|J|=d_i}u_{i,J}\boldsymbol{t}^{J}$$
for $i=0,\ldots,n$ where $J$ is a multi-index $(j_0,\ldots,j_n)$ and  $\boldsymbol{t}^{J}:=t_0^{j_0}\cdots t_n^{j_n}$ is the monomial of degree $|J|=\sum_{i=0}^{n}j_i$. The set of such polynomials form an affine space by identifying $\sum_{|J|=d_i}u_{i,J}\boldsymbol{t}^{J}$ with the point $\boldsymbol{u}_i:=(u_{i,J})_{|J|=d_i}\in\mathbb{C}^{N(d_i)}$, where $N(d_i)=\binom{n+d_i-1}{n-1}$. Letting $N:=\sum_{i=0}^{n}N(d_i)$, recall that the \textit{incidence variety} is defined by
	
			\begin{equation*}
		\mathcal{W}=\left\lbrace  (\boldsymbol{u},\boldsymbol{t})\in   \mathbb{C} ^N\times\mathbb{P}(\Bbb{C}^n) : F_0(\boldsymbol{u}_0,\boldsymbol{t})=\dots=F_n(\boldsymbol{u}_n,\boldsymbol{t})=0  \right\rbrace.
	\end{equation*}
 We also let $\pi:\mathbb{C} ^N\times\mathbb{P}(\Bbb{C}^n)\rightarrow \mathbb{C} ^N$ be the projection onto first coordinate where $\mathbb{P}(\Bbb{C}^n)$ denotes the complex projective space. Then by Projective Extension Theorem (see eg. \cite{Cox}) the image $\pi(\mathcal{W})$ forms a variety in the affine space $\mathbb{C}^N$.
	\begin{defn} The \textit{multipolynomial resultant} $Res_{d_0,\ldots,d_n}$ is defined as the irreducible unique (up to a sign) polynomial in $\mathbb{Z}[\boldsymbol{u}_0,\ldots,\boldsymbol{u}_n]$ which is the defining equation of the variety $\pi(\mathcal{W})$. The resultant of the homogeneous polynomials $F_0,\ldots, F_n$ is the evaluation of  $Res_{d_0,\ldots,d_n}$ at the coefficients of $F_0,\ldots, F_n$ and it is denoted by $Res_{d_0,\ldots,d_n}(F_0,\ldots,F_n)$. 
		\end{defn}
Note that if $d_0=\ldots=d_n=1$, then the evaluation of multipolynomial resultant $Res_{d_0,\ldots,d_n}$ at the coefficients of $F_0,\ldots,F_n$ is  the determinant of the coefficient matrix. 
	
	\begin{thm}[\cite{GKZ},\cite{Cox}]\label{th:hr}
		Let $F_0,\ldots,F_n\in\mathbb{C}[t_0,\ldots,t_n]$ be homogenous polynomials of positive total degrees $d_0,\ldots,d_n$. Then the system   $F_0=\ldots=F_n=0$ has a  solution in the complex projective space $\mathbb{P}^n$ if and only if $\text{Res}_{d_0\ldots,d_n}(F_0,\ldots,F_n)=0$. 
	\end{thm}

	Theorem \ref{th:hr} gives a characterization to determine the existence of nontrivial solutions for the systems of homogenous polynomials based on the coefficients of the polynomials in the system. However, not all the systems of equations are homogenous, and in the power series expansions not all the monomial terms appear. Hence, we need to introduce a more general version of the multi-homogenous resultant.
	
	\subsubsection{Sparse Eliminant} Following \cite{GKZ}, we will recall the definition of sparse resultant. Let $A_0,\ldots,A_n$ be a collection of non-empty finite subsets of $\mathbb{Z}^n$, and let $\boldsymbol{u}_i:=\{u_{i,J}\}_{J\in A_i} $ be a group of $\#A_i$ variables, $i=0,\ldots,n$ and set $\boldsymbol{u}=\{\boldsymbol{u}_0,\ldots,\boldsymbol{u}_n\}$ . For each $i$,  the general Laurent polynomial $f_i$  with support $A_i:=supp(f_i)$ is given by
	\begin{equation*}
		f_i(\boldsymbol{u}_i, \boldsymbol{x})=\sum_{J\in A_i}u_{i,J}\boldsymbol{x}^{J}\in\mathbb{C}[\boldsymbol{u}][x_1^{\pm1},\ldots,x_n^{\pm1}].
	\end{equation*}	
We let $\mathcal{A}=(A_0,\ldots,A_n)$ and consider the incidence variety in this setting defined by
	\begin{equation}
		W_{\mathcal{A}}=\left\lbrace  (\boldsymbol{u},\boldsymbol{x})\in  \prod_{i=0}^n\mathbb{P}(\mathbb{C}^{N_i })\times\left( \mathbb{C}^*\right) ^n : f_0(\boldsymbol{u}_1,\boldsymbol{x})=\dots=f_n(\boldsymbol{u}_n,\boldsymbol{x})=0  \right\rbrace
	\end{equation}
	where $N_i=\#A_i$. Next, we consider the canonical projection on the first coordinate $$\pi_{\mathcal{A}}: \prod_{i=0}^n\mathbb{P}(\mathbb{C}^{N_i })\times\left( \mathbb{C}^*\right) ^n \to \prod_{i=0}^n\mathbb{P}(\mathbb{C}^{N_i})$$ and let $\overline{\pi_{\mathcal{A}}(W_{\mathcal{A}})}$ denote the Zariski closure of  $W_{\mathcal{A}}$ under the projection $\pi_\mathcal{A}$.  
\begin{defn}\label{def:sparse}
	The sparse eliminant, denoted by $\text{Res}_{\mathcal{A}}$, is defined as follows: if the variety $\overline{\pi_{\mathcal{A}}(W_{\mathcal{A}})}$  has codimension 1, then the \textit{sparse eliminant} is the unique (up to sign) irreducible polynomial in $\mathbb{Z}[\boldsymbol{u}]$ which is the defining equation of $\overline{\pi_{\mathcal{A}}(W_{\mathcal{A}})}$. If $codim(\overline{\pi_{\mathcal{A}}(W_{\mathcal{A}})})\geq2$, then $\text{Res}_{\mathcal{A}}$ is defined to be the constant polynomial 1. The expression 
	$$\text{Res}_{\mathcal{A}}(f_0,\ldots,f_n)$$
	is the evaluation of $\text{Res}_{\mathcal{A}}$ at the coefficients of $f_0,\ldots,f_n$.
\end{defn}
	
	\begin{example}
		For $A_0=\left\lbrace 0\right\rbrace , A_1=\left\lbrace 0,1\right\rbrace \subset\mathbb{Z}$, we have that $\text{Res}_{\mathcal{A}}(\boldsymbol{u})=\pm u_{00}$ where $\mathcal{A}=(A_0,A_1)$.
	\end{example}
	
	
The classical resultant $Res_{d_0,\ldots,d_n}$ is the special case of the sparse eliminant $\text{Res}_{\mathcal{A}}$. Indeed, by  letting $A_i$ be the set of all integer points in the $d_i$-simplex, i.e., $A_i=d_i\Sigma_n\cap \mathbb{Z}^n$ and $\Sigma_n$ be the standard unit simplex
$$d_i\Sigma_n:=\{(a_0,\ldots,a_n)\in\mathbb{R}^{n+1}: a_j\geq 0\ \text{and}\ \sum_{j}a_j\leq d_i \}$$
one recovers $\text{Res}_{\mathcal{A}}=Res_{d_0,\ldots,d_n}$ up to a sign. Indeed, following \cite{Cox} and \cite{GKZ} for simplicity we let all the sparse polynomials $f_0,\ldots,f_n$ have the same support $A_i=d\Sigma_n\cap\mathbb{Z}^n$ for some positive integer $d$ and consider the system
\begin{equation}\label{eq:ssystem}
	\left\lbrace 
	\begin{array}{ll}
		f_0 = u_{01}\boldsymbol{x}^{\alpha_1}+\ldots+u_{0d}\boldsymbol{x}^{\alpha_n}=0\\

		\vdots\\
		f_n = u_{n1}\boldsymbol{x}^{\alpha_1}+\ldots+u_{nd}\boldsymbol{x}^{\alpha_n}=0\\

	\end{array} 
	\right. 
\end{equation}
We also let $t_0,\ldots,t_n$ be the homogenous coordinates which are related to $x_1,\ldots,x_n$ by $x_i=t_i/t_0$. Then we define the homogenous polynomials
	\begin{equation}\label{eq:hom}
		F_i(t_0,\ldots,t_n)=t_0^df_i(t_1/t_0,\ldots,t_n/t_0)=t_0^df_i(x_1,\ldots,x_n),
	\end{equation}
	for $0\leq i\leq n$. This gives n+1 homogenous polynomials of total degree $d$ in the variables $t_0,\ldots,t_n$ and this procedure is independent of the choice of homogeneous coordinates.

\begin{proposition}[\cite{Cox}] Let  $A_i:=d\Sigma_n\cap\mathbb{Z}^n$ for each $i=1,\ldots,n$ and consider the systems of polynomials $\boldsymbol{F}$ and $\boldsymbol{f}$ as above. Then

	$$Res_{\mathcal{A}_d}(f_0,\ldots,f_n)=\pm Res_{d,\ldots,d}(F_0,\ldots, F_n),$$
	where $\mathcal{A}_d:=(A_1,\ldots,A_n).$
\end{proposition}
	

	Using the above proposition, we can give a version of Theorem \ref{th:hr} as follows.

	
	\begin{cor}\label{cor:res} Let $\boldsymbol{f}=(f_1,\ldots,f_n)$ be a system of  polynomials with $A_i=d\Sigma_n\cap\mathbb{Z}^n$ for $i=1,\ldots,n$. Assume that the system $\boldsymbol{F}=(F_0,\ldots,F_n)$ consists the homogenizations of $f_i$ according to process in (\ref{eq:hom}) and denote the set of  simultaneous nonzero solutions of $\boldsymbol{F}$ by $Z(\boldsymbol{F})$. Suppose that $Z(\boldsymbol{F})\cap H^{\infty}(t_0)=\emptyset$ where $H^{\infty}(t_0)$ is the hyperplane at infinity for $t_0=0$. Then the system of  polynomials $\boldsymbol{f}=0$ has no solution if and only if $\text{Res}_{\mathcal{A}_d}(f_0,\ldots,f_n)\neq 0$ 	where $\mathcal{A}_d:=(A_1,\ldots,A_n).$
	\end{cor}
	
	\begin{proof} 	If $\text{Res}_{\mathcal{A}_d}(f_0,\ldots,f_n)\neq 0$, then by definition of the sparse resultant the system
		$$f_0(x)=\ldots=f_n(x)=0$$
		has no solution. Conversely, letting $F_i$ be the homogenization of $f_i$ as in (\ref{eq:hom}) with the corresponding variable $\boldsymbol{t}=(t_0,\ldots,t_n)$, i.e. $F_i(\boldsymbol{t})=t_0^df_i(\boldsymbol{x})$. If the system of  polynomials $\boldsymbol{f}=0$ has no solution then $F_i(\boldsymbol{t})=0$ for $i=1,\dots,n$ if and only if  $t_0=0$ which contradicts our assumption. Hence, by Theorem \ref{th:hr} we have $$\pm Res_{\mathcal{A}_d}(f_0,\ldots,f_n)=\text{Res}_{d_0,\dots,d_n}(F_0,\ldots,F_n)\neq0.$$
		
	\end{proof} 
	\subsubsection{ Sparse Resultant}
	In spite of being a generalization of the multipolynomial resultant and involving considerable large amount of the system of polynomials, the sparse eliminant does not satisfy some essential properties such as additivity property and Poisson formula which are essential in many applications. More recently,  D'Andrea and Sombra \cite{DS} introduced the following version which has the desired features: 
	
	\begin{defn}
The \emph{sparse resultant}, denoted by $\mathcal{R}es_{\mathcal{A}}$, is defined as any primitive polynomial in $\mathbb{Z}[\boldsymbol{u}]$ that is the defining equation of the direct image of $\mathcal{W}_{\mathcal{A}}$ where $$(\pi_{\mathcal{A}})_*(W_{\mathcal{A}}) = \text{deg}(\pi_{\mathcal{A}}|_{\mathcal{W}_{\mathcal{A}}})\overline{\pi_{\mathcal{A}}(\mathcal{W}_{\mathcal{A}})}$$ if this variety has codimension one, and otherwise we set $\mathcal{R}es_{\mathcal{A}}\equiv1$. The expression 
		$$\mathcal{R}es_{\mathcal{A}}(f_0,\ldots,f_n)$$
		is the evaluation of $\mathcal{R}es_{\mathcal{A}}$ at the coefficients of $f_0,\ldots,f_n$.
	\end{defn}

 According to this definition, the sparse resultant is not irreducible but it is a power of the irreducible sparse eliminant, i.e.,
	
	\begin{equation*}
		\mathcal{R}es_{\mathcal{A}}=\pm \text{Res}_{\mathcal{A}}^{\text{deg}(\pi_{\mathcal{A}}|_{W_{\mathcal{A}}})}
	\end{equation*}
	
	\noindent where $\text{deg}(\pi_{\mathcal{A}}|_{W_{\mathcal{A}}})$ is the degree of the projection $\pi_{\mathcal{A}}$. We also remark that $\mathcal{R}es_{\mathcal{A}}\not\equiv1$ whenever $\text{Res}_{\mathcal{A}}\not\equiv1$. 
	
	\begin{example}
		Let $A_0=A_1=A_2=\{(0,0),(2,0),(0,2)\}$. Then $\text{Res}_{\mathcal{A}}= \det(u_{i,j})$ and $\mathcal{R}es_{\mathcal{A}}=\pm [\det(u_{i,j})]^4$.
	\end{example}	
	For the detailed account of the subject we refer the reader to the manuscripts \cite{DS} and \cite{DGS}.
	
\subsubsection{Directional Resultant} For a finite subset $A\subset\mathbb{Z}^n$ and a non-zero vector $\boldsymbol{v}\in\mathbb{Z}^n$ we denote 
$$\mathcal{A}^{\boldsymbol{v}}:=\left\lbrace J\in A: \left\langle J,\boldsymbol{v}\right\rangle = s_{Q}(\boldsymbol{v}) \right\rbrace $$
where $Q=conv(A)$ and  $s_{Q}(\boldsymbol{v})$ as in the equation (\ref{support function}). 
For a Laurent polynomial $f(x)=\sum_{J\in A}u_{J}\boldsymbol{x}^{J}$ with support $supp(f)=A$ we also define the directed polynomial
$$f^{\boldsymbol{v}}(x):=\sum_{J\in\mathcal{A}^{\boldsymbol{v}}}u_{J}\boldsymbol{x}^{J}.$$

	\begin{defn}{\label{def:directional}}
		Let $A_1,\ldots,A_n\subset\mathbb{Z}^n$ be a family of $n$ non-empty finite subsets,  $\boldsymbol{v}\in\mathbb{Z}^n\setminus\left\lbrace \boldsymbol{0}\right\rbrace$, and $\boldsymbol{v}^\perp\subset\mathbb{R}^n$ the orthogonal subspace. Then there exists $\boldsymbol{b}_{i,\boldsymbol{v}}\in\mathbb{Z}^n$ such that 
		$$A_i^{\boldsymbol{v}}-\boldsymbol{b}_{i,\boldsymbol{v}}\subset\mathbb{Z}^n\cap\boldsymbol{v}^\perp$$
		 for $i=1,\ldots,n$. The resultant of $A_1,\ldots,A_n$ in the direction of $\boldsymbol{v}$, denoted by $\mathcal{R}es_{\mathcal{A}^{\boldsymbol{v}}}$ is defined as the sparse resultant of the family of the finite subsets $A_i^{\boldsymbol{v}}-\boldsymbol{b}_{i,\boldsymbol{v}}$ for $i\in\{1,\ldots,n\}$.
		
		Given a collection $f_i\in\mathbb{C}[x_1^{\pm1},\ldots,x_n^{\pm1}]$ of Laurent polynomials with support $supp(f_i)\subset A_i$ for $i=1,\ldots,n$ we write $f_i^{\boldsymbol{v}}=\boldsymbol{x}^{\boldsymbol{b}_{i,\boldsymbol{v}}}g_{i,\boldsymbol{v}}$ where $g_{i,\boldsymbol{v}}\in\mathbb{C}[\mathbb{Z}^n\cap\boldsymbol{v}^{\perp}]\simeq\mathbb{C}[y_1^{\pm1},\ldots,y_{n-1}^{\pm1}]$ is a Laurent polynomial with $supp(g_{i,\boldsymbol{v}})\subset A_i^{\boldsymbol{v}}-\boldsymbol{b}_{i,\boldsymbol{v}}$. The expression
		$$\mathcal{R}es_{\mathcal{A}^{\boldsymbol{v}}}(f_1^{\boldsymbol{v}},\ldots,f_n^{\boldsymbol{v}})$$
		is defined as the evaluation of the resultant $\mathcal{R}es_{\mathcal{A}^{\boldsymbol{v}}}$ at the coefficients of the $g_{i,\boldsymbol{v}}$. 
	\end{defn}
We remark that the definition of directional resultant is independent of the choice of the vector $\boldsymbol{b}_{i,\boldsymbol{v}}$ (see \cite[Proposition 3.3]{DS}). Moreover, the directional resultant $\mathcal{R}es_{\mathcal{A}^{\boldsymbol{v}}} \not \equiv1$ only if the direction vector $\boldsymbol{v}$ is an inward pointing normal to a facet of the Minkowski sum $\sum_{i=1}^n conv(A_i)$ (cf. \cite[Proposition 3.8]{DS}). Therefore, for a family of subsets $A_1,\ldots,A_n\subset \Bbb{Z}^n$ there are only finitely many directions $\boldsymbol{v}\in\mathbb{Z}^n\setminus\left\lbrace \boldsymbol{0}\right\rbrace$ for which the directional resultant can vanish.

\begin{example} Let $f(x)=a_0+\ldots+a_nx^n\in\mathbb{C}[x]$ be a polynomial of degree $n$. Then the nontrivial directional resultants are 
	\begin{equation*}
		\mathcal{R}es_{A}(f^{\boldsymbol{v}})= \left\lbrace 
		\begin{array}{ll}
			\pm a_0& \text{if}\quad\boldsymbol{v}=1, \\

			\pm a_n & \text{if}\quad \boldsymbol{v}=-1 \\
		\end{array} 
		\right. 
	\end{equation*}
	
	\noindent for the polytope $conv(A)=[0,n]\subset\mathbb{R}$.
	\end{example} 

In the last part of this section we review Bernstein's Theorem about the number of the common solutions for Laurent polynomial systems and its relation to the directional resultant. The classical B\'{e}zout's Theorem states  that for $n$ polynomials $g_1,\ldots,g_n\in \mathbb{C}[x_1,\ldots,x_n]$ of (positive) degrees $d_1,\ldots,d_n$ the system $$g_1(x_1,\dots, x_n) = \cdots=g_n(x_1,\dots, x_n)=0$$ has either infinite number of solutions or the number of the number of complex roots cannot exceed $d_1\ldots d_n$. Moreover, if the solutions in the hyperplane at infinity are counted with multiplicity, the exact number of solutions in the complex projective space $\mathbb{P}^{n}$ is $d_1\cdots d_n$ (see e.g. \cite{Cox}). A generalization of this result to the context of Laurent polynomials was obtained by Bernstein \cite{Ber} (see also Kushnirenko \cite{Kush}). More precisely, we have the following:

\begin{thm}[\cite{Ber}]\label{BerTh}
		Let $\boldsymbol{f}=(f_1,\ldots,f_n)$ be a system of Laurent polynomials with \linebreak support $supp(f_i)=A_i\subset\Bbb{Z}^n$ for $i=1,\ldots,n$. If for any nonzero vector $\boldsymbol{v}\in\Bbb{Z}^n$ the directed system $\boldsymbol{f}^{\boldsymbol{v}}=(f_{1}^{\boldsymbol{v}},\ldots,f_n^{\boldsymbol{v}})$ has no common zeros in $(\mathbb{C}^*)^n$  then the set of solutions of the system $\boldsymbol{f}=0$ are isolated and the exact number of the solutions is $\#Z(\boldsymbol{f})=MV_{n}(Q_1,\ldots,Q_n)$  where $Q_i=conv(A_i)$ for $i=1,\ldots,n.$
		\end{thm}
	In particular, for a system of Laurent polynomials $\boldsymbol{f}=(f_1,\ldots,f_n)$ if the directional resultant $\mathcal{R}es_{\mathcal{A}^{\boldsymbol{v}}}(f_1^{\boldsymbol{v}},\ldots,f_n^{\boldsymbol{v}})\not=0$ for every direction $\boldsymbol{v}\in\Bbb{Z}^n\setminus\{0\}$ then simultaneous solutions of the system $\boldsymbol{f}$ are isolated. This condition on the directional resultant holds for a generic (i.e. all except for some algebraic subset) choice of $\boldsymbol{f}$ in the space of coefficients. In the next section, we prove a probabilistic version of this result for polynomial systems with Bernoulli coefficients.


\section{Equidistribution of Zeros}
	
\subsection{Random Polynomial Systems}
First, we recall a theorem of Kozma and Zeitouni \cite{KZ} asserts that overdetermined random Bernoulli polynomial systems have no common zeros with overwhelming probability:

\begin{thm}\label{th:KZ}
	Let $f_1,\ldots,f_{n+1}\in\mathbb{Z}[x_1,\ldots,x_n]$ be $n+1$ independent random Bernoulli polynomials of degree $d$ and
	\begin{equation*}	\mathcal{P}(d,n):=Prob_d\{\exists\boldsymbol{x}\in\mathbb{C}^n: f_i(\boldsymbol{x})=0\ \text{for}\ i=1\ldots,n+1\}
	\end{equation*}
	denote the probability that the system $f_1(\boldsymbol{x})=\ldots=f_{n+1}(\boldsymbol{x})=0$ has a common solution. Then there exists a dimensional constant $K=K(n) <\infty$ such that $$\mathcal{P}(d,n)\leq K/d$$ for all $d\in\mathbb{Z}_+$.
\end{thm}

Next, we prove our main result:
\begin{proof}[Proof of Theorem \ref{th:main}]
Let $f_{d,i}$ be a random Bernoulli polynomial of the form
	\begin{equation}
		f_{d,i}(\boldsymbol{x})=\sum_{|J|\leq d}\alpha_{i,J}\boldsymbol{x}^{J} \in\mathbb{Z}[x_1,\ldots,x_n],
	\end{equation}
where $\{\alpha_{i,J}\}$ is a family of independent Bernoulli random variables for $i=1,\ldots,n$. 

		We investigate the directional resultants of the system $\boldsymbol{f}$ for all nonzero primitive direction vectors $\boldsymbol{v}\in\mathbb{Z}^n$. 
	By  \cite[Proposition 3.8]{DS} it is enough to check the inward normals to the Minkowski sum of the supports $nd\Sigma_n$ which has $n+1$ facets with $n+1$ inward normals given by  $\boldsymbol{v}_m:=\boldsymbol{e}_m$ for $m=1,\ldots,n$ and $\boldsymbol{v}_{n+1}:=-\sum_{m=1}^n\boldsymbol{e}_m$ where
	$\{\boldsymbol{e}_m\}_{m=1}^n$ is the standard basis of $\mathbb{R}^n$. 
	
	For $\boldsymbol{v}_m=\boldsymbol{e_m}$ the intersection of the support with the supporting hyperplane in the direction $\boldsymbol{e_m}$ is of the form 
	
	\begin{equation}	\mathcal{A}^{\boldsymbol{v}_m}=\left\lbrace(j_1,\ldots,j_n)\in d\Sigma_n\cap \mathbb{Z}^n: j_m=0\ \text{and}\ \sum_{l=1}^n j_l \leq d	\right\rbrace
	\end{equation}
	$m=1,\ldots,n$. 
	Hence, the polynomials $f_i^{\boldsymbol{v}_m}$ can be written as  
	\begin{equation}
		f_i^{\boldsymbol{v}_m}:= \sum_ {J\in\mathcal{A}^{\boldsymbol{v}_m}}\alpha_{i,J}\boldsymbol{x}^{J}
	\end{equation}
	for $i=1,\ldots,n$. Note that polynomials $f_i^{\boldsymbol{v}_m}$ depend on $n-1$ variables.  As in the Definition \ref{def:directional}, we choose  the vector  $\boldsymbol{b}_{i,\boldsymbol{v}_m}=\mathbf{0}$ so that $\mathcal{A}^{\boldsymbol{v}_m}-\boldsymbol{b}_{i,\boldsymbol{v}_m}\subset\mathbb{Z}^n\cap{\boldsymbol{v}_m}^{\perp}$ and we may take $g_{i,\boldsymbol{v}_m}:=f_i^{\boldsymbol{v}_m}$ for each $i=1,\ldots,n$.
	
	Recall that for two univariate polynomials $h_1,h_2\in\mathbb{C}[x]$, their resultant $\mathcal{R}es(h_1,h_2)$ is zero if and only if $h_1$ and $h_2$ have a common solution in  $\mathbb{C}$. Therefore, if $n=2$ the necessary and sufficient condition for $g_{1,\boldsymbol{v}_m}$ and $g_{2,\boldsymbol{v}_m}$ have zero resultant is that they have a common zero. Theorem \ref{th:KZ} implies that there exists a constant $K_m$ which is independent of $d$ so that the aforementioned event has probability at most $K_m/d$.
	
	On the other hand, when $n>2$, we perform the homogenization process to each $(n-1)$ variable polynomial $g_{i,\boldsymbol{v}_m}$ for $i=1,\ldots,n$ as described in equation (\ref*{eq:hom}). We obtain the $n$ variable homogenous polynomials $G_{i,\boldsymbol{v}_m}$ of the form 
	
	\begin{equation}
		G_{i,\boldsymbol{v}_m}(t,\boldsymbol{x})=\sum_{\textbf{$J$}\in\mathcal{A}^{\boldsymbol{v}_m}}\alpha_{i,J}t^{d-|J|}\boldsymbol{x}^{J}.
	\end{equation}
In order to compare the sparse resultant of the polynomials $g_{i,\boldsymbol{v}_m}$ and the multipolynomial resultant of the  homogeneous polynomials $G_{i,\boldsymbol{v}_m}$, we check the conditions of Corollary \ref*{cor:res}. Let $Z(\boldsymbol{G})$ be the set of nontrivial solutions of the system $\boldsymbol{G}=(G_{1,\boldsymbol{v}_m},\ldots,G_{n,\boldsymbol{v}_m})$ and suppose that $\boldsymbol{G}$ has a solution $\boldsymbol{\xi}=(t,\xi_2,\ldots,\xi_n)$ in the hyperplane at infinity $H^{\infty}(t)$. Evaluating these homogeneous polynomials at $t=0$, we obtain the top degree homogeneous part of the polynomials $g_{i,\boldsymbol{v}_m}$ for $i=1,\ldots,n$. Since $\boldsymbol{\xi}\in H^{\infty}(t)$, it has a nonzero coordinate $\xi_k$ for some $k\in\{2,\ldots,n\}$. For simplicity, let us assume $k=2$ and  define the new variables  $z_{i}:={\xi_{i+2}}/\xi_2$ for $i=1,\ldots,n-2$. Applying this change of variables, we obtain 
	
	\begin{equation}
		\widetilde{G}_{i,\boldsymbol{v}_m}(z_1,\ldots,z_{n-2})=\sum_{|J|\leq d}\alpha_{i,J}\boldsymbol{z}^{\varphi(J)}
	\end{equation}
	where $\varphi:\mathbb{R}^n\rightarrow\mathbb{R}^{n-2}$ with $\varphi(j_1,\ldots,j_n)=(j_3,\ldots,j_n)$. This gives $n$ random Bernoulli  polynomials of degree $d$ in $n-2$ variables. Hence by Theorem \ref{th:KZ}, there exists a positive constant $C_i$, depending only the dimension $n$ such that the probability that the overdetermined system of Bernoulli polynomials $\widetilde{G}_{i,\boldsymbol{v}_m}(z_1,\ldots,z_{n-2})$ have a common solution is less than $C_i/d$.  We infer that the system of homogenized polynomials $G_{i,\boldsymbol{v}_m}$ has no common zero at hyperplane at infinity $H^{\infty}(t)$ except a set  that has  probability at most $C_i/d$.  Then by Corollary \ref{cor:res}, outside of a set of small probability,  the system of polynomials consisting $g_{i,\boldsymbol{v}_m}$ has a common solution if and only if the directional resultant $\mathcal{R}es_{\mathcal{A}^{\boldsymbol{v}_m}}(f_{1}^{\boldsymbol{v}_m},\ldots,f_{n}^{\boldsymbol{v}_m})=0$. Now, since  the system of Bernoulli polynomials $g_{i,\boldsymbol{v}_m}$ contains $n$ polynomials in $n-1$ variables, by Theorem \ref{th:KZ}, there is a dimensional constant $\widetilde{C}_i$ so that the probability that this system has common solution is at most $\widetilde{C}_i/d$. Hence outside of a set that has probability $K_i/d:=C_i/d+\widetilde{C_i}/d$ , the directional resultant $\mathcal{R}es_{\mathcal{A}^{\boldsymbol{v}_m}}(f_{1}^{\boldsymbol{v}_m},\ldots,f_{n}^{\boldsymbol{v}_m})\neq0$ for all $\boldsymbol{v}_m$  for $m=1,\ldots,n$.
	
	Next, we consider the inward normal vector $\boldsymbol{v}_{n+1}=-\sum_{m=1}^n \boldsymbol{e}_m$ and we find the minimal weighted set in this direction as $\mathcal{A}^{\boldsymbol{v}_{n+1}}=\left\lbrace J\in d\Sigma_n\cap \mathbb{Z}^n: |J|=d\right\rbrace$. Hence, the directed polynomials in this case are of  the form 
	\begin{equation}
		f_i^{\boldsymbol{v}_{n+1}}(\boldsymbol{x})=\sum_{|J|=d}\alpha_{i,J}\boldsymbol{x}^{J}
	\end{equation}
	In this case $\mathcal{A}^{\boldsymbol{v}_{n+1}}$ is not a subspace of $\mathbb{Z}^n\cap\boldsymbol{v}_{n+1}^{\perp}$, hence we need to translate it by subtracting a suitable vector $\boldsymbol{b}_{i,\boldsymbol{v}_{n+1}}$. For Laurent polynomial systems, the sparse resultant is invariant under translations of supports (see \cite{DS}, Proposition 3.3). Since the polynomials $f_{d,i}$ are not Laurent, we need to determine the effects of this translations. Consider the system of Bernoulli polynomials $\boldsymbol{f}_d$ and set of its simultaneous zeros $Z(\boldsymbol{f}_d)$. For a solution $\boldsymbol{x}=(x_1,\ldots,x_n)\in Z(\boldsymbol{f}_d)$ and assume that $x_1=0$. In order to examine the incidence of this case, we evaluate the system $\boldsymbol{f}_d$ at $x_1=0$ and we obtain a new system of $n$ Bernoulli polynomials with $n-1$ variables. By Theorem \ref{th:KZ}, there exists a constant $C_1$ which is independent of $d$ such that this system has a common solution with probability at most $C_1/d$. Therefore the probability of the event that $x_1=0$ is less than $C_1/d$. Hence, there is no harm of translation of supports outside of a set that has probability at most $C/d$, where $C:=\sum_{i=1}^n C_i$. Now, choosing the vector  $\boldsymbol{b}_{i,\boldsymbol{v}_{n+1}}=(d,0,\ldots,0)$ so that $\mathcal{A}^{\boldsymbol{v}_{n+1}}-\boldsymbol{b}_{i,\boldsymbol{v}_{n+1}}\subset\mathbb{Z}^n\cap\boldsymbol{v}_{n+1}^{\perp}$, we obtain the polynomials of the form 
	\begin{equation}
		g_{i,\boldsymbol{v}_{n+1}}(\boldsymbol{x})=\sum_{J\in\mathcal{A}^{\boldsymbol{v}_{n+1}}-\boldsymbol{b}_{i,\boldsymbol{v}_{n+1}}}\alpha_{i,J}\boldsymbol{x}^{w(J)}
	\end{equation}
	with $w:\mathbb{R}^n\rightarrow\mathbb{R}^n$ satisfying $(j_1,j_2,\ldots,j_n)\mapsto(-d+j_1,j_2,\ldots,j_n)$. We substitute the new variables $y_i:=x_{i+1}/x_1$ into $g_{i,\boldsymbol{v}_{n+1}}$ for $i=1,\ldots,n-1$ and obtain
	
	\begin{equation}
		g_{i,\boldsymbol{v}_{n+1}}(\boldsymbol{y})=\sum_{|J|\leq d}\alpha_{i,J}\boldsymbol{y}^{\sigma(J)}
	\end{equation}
	for $\boldsymbol{y}\in\mathbb{C}^{n-1}$ and $\sigma:\mathbb{R}^n\rightarrow\mathbb{R}^n$ with $\sigma(j_1,j_2,\ldots,j_n)=(0,j_2,\ldots,j_n)$. The system containing the polynomials 	$g_{i,\boldsymbol{v}_{n+1}}(\boldsymbol{y})$, $i=1,\ldots,n$ contains $n$ random Bernoulli polynomials with $n-1$ random variable as in the cases $\boldsymbol{v}_m={\bf e}_m$. By applying the same argument, we can show that $\mathcal{R}es_{\mathcal{A}^{\boldsymbol{v}_{n+1}}} (f_{1}^{\boldsymbol{v}_{n+1}},\ldots,f_{n}^{\boldsymbol{v}_{n+1}})\neq0$ outside of a set that has probability at most $K_{i+1}/d$.

Now, we define the exceptional set $\mathcal{E}_{n,d}$ as a subset of $Poly_{n,d}$ which contains the systems $\boldsymbol{f}_d$ that has a zero directional resultant for some nonzero primitive vector $\boldsymbol{v}$ or the systems $\boldsymbol{f}_d$ have a common solution $\boldsymbol{x}\in\mathbb{C}^n$ with $x_i=0$ for some $i=1,\ldots,n$. More precisely, letting
\begin{eqnarray}\label{exceptional}
\mathcal{E}_{n,d}&:=&\{\boldsymbol{f}_d\in Poly_{n,d}: \exists~\boldsymbol{v}\in\mathbb{Z}^n\setminus\{\boldsymbol{0} \} \ni \mathcal{R}es_{\boldsymbol{A}^{\boldsymbol{v}}}\boldsymbol{f}_d^{\boldsymbol{v}}=0\}\\ \nonumber
	&\bigcup&\{\boldsymbol{f}_d\in Poly_{n,d}: \exists~\boldsymbol{x}\in Z(\boldsymbol{f}_d) \ni \prod x_i=0\}
\end{eqnarray}
we see that there exists a positive constant $K$ which is independent of $d$ such that $$Prob_d(\mathcal{E}_{n,d})\leq \frac{K}{d}$$ where $K:=\sum_{i=1}^{n+1}K_i+C$. 
\end{proof}

Next, we recall a deterministic equidistribution results for the solutions of systems of integer coefficient polynomials \cite{DGS}. 	
	For a polynomial $f\in\mathbb{C}[x_1,\ldots,x_n]$, the supremum norm of $f$ on the unit torus is defined as 
	\begin{equation*}
		\left\| f\right\| _{\sup}:= \sup_{|w_1|=\ldots=|w_n|=1}\left| f(w_1,\ldots,w_n)\right|.
	\end{equation*}
Let $\nu_{\text{Haar}}$ be the Haar measure on $\mathbb{C}^n$ with support $(S^1)^n$ and of total mass 1. Assume that $\boldsymbol{f} \in Poly_{n,d}$ be a polynomial mapping such that the set of simultaneous zeros $Z(\boldsymbol{f})$ is a discrete set. We denote by denote the discrete probability measure on $\mathbb{C}^n$ associated to the $Z(\boldsymbol{f})$ by $\delta_{Z(\boldsymbol{f})}$. The following result gives the asymptotic distribution of the zeros of such a system $\boldsymbol{f}$ if the coefficients are integer: 

\begin{thm}\label{th:dgs}\cite{DGS} Let $ \boldsymbol{f}=(f_1,\ldots,f_n)$ be a polynomial mapping with $f_i\in \mathbb{Z}[x_1,\ldots,x_n]$ of degree $d\geq1$ for each $i=1,\dots,n$. Assume that $\mathcal{R}es_{\mathcal{A}^{\boldsymbol{v}}}(f_1^{\boldsymbol{v}},\ldots,f_n^{\boldsymbol{v}})\neq0$ for all $v\in\mathbb{Z}^n\setminus\{\mathbf{0}\}$ and $\log||f_i||_{\sup}=o(d)$. Then 
	
	\begin{equation*}
		\lim_{d\rightarrow\infty}\delta_{Z(\boldsymbol{f})}= \nu_{\text{Haar}}
	\end{equation*}
	in the weak topology.
\end{thm}

\begin{proof}[Proof of Corollary \ref{cor:eq}]
Consider the system of Bernoulli polynomials $\boldsymbol{f}_d=(f_{d,1},\ldots,f_{d,n})$. Since all the coefficients are $1$ or $-1$, by triangle inequality

	\begin{equation}
	\left\| f_{d,i}\right\|_{\sup}= \sup_{|w_1|=\ldots=|w_n|=1}|f_{d,i}(w_1,\ldots,w_n)|\leq {{n+ d}\choose{d}}= d^n+O(d^{n-1})
	\end{equation}
	where ${{n+ d}\choose{d}}$ is the dimension of space of polynomials $Poly_{n,d}$. 
This in turn implies that  $\log\left\| f_{d,i}\right\|_{\sup} = o(d) $. Moreover, by Theorem \ref{th:main} for each $\boldsymbol{f}_d\in Poly_{n,d}\setminus\mathcal{E}_{n,d}$ we have $$\mathcal{R}es_{\mathcal{A}^{\boldsymbol{v}}}(f_1^{\boldsymbol{v}},\ldots,f_n^{\boldsymbol{v}})\neq0$$ for all $v\in\mathbb{Z}^n\setminus\{\mathbf{0}\}$. Hence, by Theorem \ref{th:dgs} 
	\begin{equation*}
	\lim_{d\rightarrow\infty}\delta_{Z(\boldsymbol{f}_d)}= \nu_{\text{Haar}}
\end{equation*}
in the weak topology. In particular,  $\delta_{Z(\boldsymbol{f}_d)}\rightarrow \nu_{\text{Haar}}$ in probability  since $Prob_d(\mathcal{E}_{n,d})\to 0$ as $d\to\infty$.
  		
\end{proof}

\section{Expected Zero Distribution }

In this section, we introduce radial and angle discrepancies for random Bernoulli polynomial mappings in order to study asymptotics of expected zero measures. We adapt these concepts from \cite{DGS} and refer the reader to the manuscript \cite{DGS} and references therein for a detailed account of the preliminary results this section.

Let $Z$ be a 0-dimensional effective cycle in $\Bbb{C}^n$ that is there is a non-empty finite collection of points $\boldsymbol{\xi}=(\xi_1,\ldots,\xi_n)\in\mathbb{C}^n$ and $m_{\boldsymbol{\xi}}\in\mathbb{N}$, called the multiplicity of $\boldsymbol{\xi}$, such that $ Z=\sum_{\boldsymbol{\xi}}m_{\boldsymbol{\xi}}[{\boldsymbol{\xi}}]$. The degree of $Z$ is defined by $deg(Z)=\sum_{\boldsymbol{\xi}}m_{\boldsymbol{\xi}}$ which is a positive number.

\begin{defn}\cite{DGS}
	Let $Z$ be a 0-dimensional effective cycle in $\Bbb{C}^n$. For each $\boldsymbol{\alpha}=(\alpha_1,\ldots,\alpha_n)$ and $\boldsymbol{\beta}=(\beta_1,\ldots,\beta_n)\in \Bbb{R}^n$ such that $-\pi\leq\alpha_j<\beta_j\leq\pi$, $j=1,\ldots,n$ we consider the cycle
	\begin{equation*}
		Z_{\boldsymbol{\alpha},\boldsymbol{\beta}}:=\sum_{\{\boldsymbol{\xi}\in Z: \alpha_j<\arg(\xi_j)\leq\beta_j \}}m_{\boldsymbol{\xi}}[{\boldsymbol{\xi}}].
		\end{equation*}
		The $\textit{angle discrepancy}$ of $Z$ is defined as 
		\begin{equation*}
			\Delta_{\text{ang}}(Z):=\sup_{\boldsymbol{\alpha},\boldsymbol{\beta}}\left| \frac{\text{deg}(Z_{\boldsymbol{\alpha},\boldsymbol{\beta}})}{\text{deg}(Z)}-\prod_{j=1}^n\frac{\beta_j-\alpha_j}{2\pi}\right|.
		\end{equation*}
For $0<\varepsilon<1$ we consider the cycle
	\begin{equation*}
		Z_{\varepsilon}:=\sum_{\{\boldsymbol{\xi}\in Z: 1-\varepsilon<|\xi_j|<(1-\varepsilon)^{-1}\}}m_{\boldsymbol{\xi}}[{\boldsymbol{\xi}}].
	\end{equation*}
	The $\textit{radius discrepancy}$ of $Z$ with respect to $\varepsilon$ is defined as 
	\begin{equation*}
		\Delta_{\text{rad}}(Z,\varepsilon):= 1- \frac{\text{deg}(Z_{\varepsilon})}{\text{deg}(Z)}.
	\end{equation*}
\end{defn}
Note that $0<\Delta_{\text{ang}}(Z)\leq1$ and $0\leq\Delta_{\text{rad}}(Z,\varepsilon)\leq1$. Observe that the angle discrepancy and the radial discrepancy are generalizations of their one dimensional versions defined in  \cite{ET,HN}. 

Let $A_1,\ldots,A_n\subset \mathbb{Z}^n$ be a collection of finite sets and let $Q_i=conv(A_i)$ for each $i=1,\ldots,n$. Throughout this section we assume that $D:=MV_n(Q_1,\ldots,Q_n)\geq1$. For a vector $\boldsymbol{w}\in S^{n-1}$ in the unit sphere in $\mathbb{R}^n$, let $\boldsymbol{w}^{\perp}$ be its orthogonal subspace and $\pi_{\boldsymbol{w}^{\perp}}:\mathbb{R}^n\rightarrow \boldsymbol{w}^{\perp}$ be the corresponding orthogonal projection. We let $MV_{\boldsymbol{w}^{\perp}}$ denote the mixed volume of the convex bodies in $\boldsymbol{w}^{\perp}$ induced by the Euclidean measure on $\boldsymbol{w}^{\perp}$. We also denote 
$$D_{\boldsymbol{w},i}=MV_{\boldsymbol{w}^{\perp}}\left( \pi_{\boldsymbol{w}}(Q_1),\ldots,\pi_{\boldsymbol{w}}(Q_{i-1}),\pi_{\boldsymbol{w}}(Q_{i+1}),\ldots,\pi_{\boldsymbol{w}}(Q_n)\right) .$$
Let $\boldsymbol{f}=(f_1,\dots, f_n)$ be a mapping such that the coordinates $f_i$ are Laurent polynomials with $supp(f_i)=A_i$ for $i=1,\dots,n$. Following \cite{DGS}, we define the \textit{Erd\"{o}s-Tur\'{a}n size } of $\boldsymbol{f}$ by 

\begin{equation}
	\eta(\boldsymbol{f}):=\frac{1}{D} \sup _{\boldsymbol{w}\in S^{n-1}}\log \left( \frac{\prod_{i=1}^{n}||f||_{sup}^{D_{\boldsymbol{w},i}}}{\prod_{\boldsymbol{v}}|\mathcal{R}es_{\mathcal{A}^{\boldsymbol{v}}}(f_1^{\boldsymbol{v}},\ldots,f_n^{\boldsymbol{v}})|^{\frac{|\left\langle \boldsymbol{v},\boldsymbol{w}\right\rangle| }{2}}}\right)	
\end{equation}
where $\left\langle \cdot,\cdot\right\rangle $ is the standard inner product in $\mathbb{R}^n$ and the product in the denominator is taken over all non-zero primitive vectors $\boldsymbol{v}\in \mathbb{Z}^n$. We remark that the Erd\"{o}s-Tur\'{a}n size of a polynomial mapping $\boldsymbol{f}$ coincides with the bound in the Erd\"{o}s-Tur\'{a}n Theorem \cite{ET} for univariate polynomials.\\
 The next result gives an  upper bound for the Erd\"{o}s-Tur\'{a}n size of polynomial systems $\boldsymbol{f}$ with integer coefficients. 

\begin{proposition}\label{prop1}\cite[Proposition 3.15]{DGS}  Let $A_1,\ldots,A_n$ be a non-empty finite subsets of $\mathbb{Z}^n$ and set $Q_i=conv(A_i)$ with $MV_n(Q_1,\ldots,Q_n)\geq1$. Let  $d_i\in\mathbb{Z}_{\geq1}$ and $\boldsymbol{b}_i\in\mathbb{Z}^n$ so that $d_i\Sigma_n+\boldsymbol{b}_i$, $i=1,\ldots,n$.	Suppose that $f_1,\ldots,f_n\in\mathbb{Z}[x_1^{\pm1},\ldots,x_n^{\pm1}]$ with $supp(f_i)\subseteq A_i$ and such that $\mathcal{R}es_{\mathcal{A}_1^{\boldsymbol{v}},\ldots,\mathcal{A}_n^{\boldsymbol{v}}}(f_{d,1}^{\boldsymbol{v}},\ldots,f_{d,n}^{\boldsymbol{v}})\neq0$ for all $\boldsymbol{v}\in\mathbb{Z}^n\setminus\{\boldsymbol{0}\}.$ Then 
	\begin{equation*}
		\eta(\boldsymbol{f})\leq \frac{1}{D}\left( \left( n+\sqrt{n}\right) \left(\prod_{i=1}^nd_i\right) \sum_{i=1}^{n}\frac{\log\left\| f_i\right\| _{\sup}}{d_i}\right) .
	\end{equation*}
\end{proposition}
The following theorem gives bounds for angle discrepancy and radius discrepancy of $Z(\boldsymbol{f})$ in terms of the Erd\"{o}s-Tur\'{a}n size of $\boldsymbol{f}$. For one dimensional version see for instance \cite{ET} and \cite{HN}.

\begin{thm}\label{th1.4}\cite{DGS} Let $A_1,\ldots,A_n$ be a non-empty finite subsets of $\mathbb{Z}^n$ such that $$MV_n(Q_1,\ldots,Q_n)\geq1$$ with $Q_i=conv(A_i)$  for $n\geq2$. Let $f_1,\ldots,f_n\in\mathbb{C}[x_1^{\pm1},\ldots,x_n^{\pm1}]$ with $supp(f_i)\subseteq A_i$ and such that $\mathcal{R}es_{\mathcal{A}^{\boldsymbol{v}}}(f_{d,1}^{\boldsymbol{v}},\ldots,f_{d,n}^{\boldsymbol{v}})\neq0$ for all $\boldsymbol{v}\in\mathbb{Z}^n\setminus\{\boldsymbol{0}\}.$ Then
	\begin{equation}\label{bound:ang}
		\Delta_{\text{ang}}(Z(\boldsymbol{f}))\leq 66n2^n(18+\log^+(\eta(\boldsymbol{f})^{-1}))^{\frac{2}{3}(n-1)}\eta(\boldsymbol{f})^{\frac{1}{3}}.
	\end{equation}\label{bound:rad}
	Moreover, for $0<\varepsilon<1$,
	\begin{equation}
		\Delta_{\text{rad}}(Z(\boldsymbol{f}), \varepsilon)\leq \frac{2n}{\varepsilon}\eta(\boldsymbol{f}).
	\end{equation}
\end{thm}
For a random Bernoulli polynomial mapping $\boldsymbol{f}_d$ we let $Z(\boldsymbol{f}_d)$ be the set of simultaneous zeros of $\boldsymbol{f}_d$. We define the angle discrepancy  $\Delta_{\text{ang}}(Z(\boldsymbol{f}))$ and the radius discrepancy $\Delta_{\text{rad}}(Z(\boldsymbol{f}), \varepsilon)$ as above 
whenever $Z(\boldsymbol{f}_d)$ is a discrete set of points. Otherwise, we set $\Delta_{\text{rad}}(Z(\boldsymbol{f}), \varepsilon)= \Delta_{\text{ang}}(Z(\boldsymbol{f}))=1$. Note that as our
probability space $(Poly_{n,d}, Prob_d)$ is discrete, measurability of these random variables is not an issue in this setting. Next, we estimate the asymptotic expected discrepancies: 
 
 \begin{proposition}\label{prop}
 Let $\boldsymbol{f}_d=(f_{d,1},\ldots,f_{d,n})$ be a random Bernoulli polynomial mapping of degree $d\geq 1$. Then 
 
 \begin{equation}
 	\lim_{d\to\infty}\mathbb{E}[\Delta_{\text{ang}}(Z(\boldsymbol{f}_d))]=0\quad\text{and}\quad \lim_{d\to\infty}\mathbb{E}[\Delta_{\text{rad}}(Z(\boldsymbol{f}_d))]=0.
 \end{equation} 
 \end{proposition}
\begin{proof} We adapt the argument in [\cite{DGS}, Theorem 4.9] to our setting. Consider the expected value of the angular discrepancy which is 
	\begin{equation}
		\mathbb{E}[Z(\boldsymbol{f}_d)]=\int_{Poly_{n,d}}\Delta_{\text{ang}}(Z(\boldsymbol{f}_d))dProb_d(\boldsymbol{f}_d).
	\end{equation}
Let $\mathcal{E}_{n,d}$ be the exceptional set which contains all the systems in $Poly_{n,d}$ with zero directional resultants for some nonzero primitive vector $\boldsymbol{v}\in\mathbb{Z}^n$ as described in the proof of Theorem 1.1. Since $0<\Delta_{\text{ang}}(Z(\boldsymbol{f}_d))\leq1$ there exist constants  $K_1$ which is independent of $d$ such that
	\begin{equation}
		0\leq\int_{\mathcal{E}_{n,d}}\Delta_{\text{ang}}(Z(\boldsymbol{f}_d))dProb(\boldsymbol{f}_d)\leq   Prob_d\{\mathcal{E}_{n,d}\}\leq  K_1 d^{-1}.
	\end{equation}
Hence, 
$$\int_{\mathcal{E}_{n,d}}\Delta_{\text{ang}}(Z(\boldsymbol{f}_d))dProb_d(\boldsymbol{f}_d)\to 0$$ as $d\to\infty$. 

Let $\boldsymbol{f}_d\in Poly_{n,d}\setminus\mathcal{E}_{n,d}$, then by Proposition \ref{prop1} 

\begin{flalign}
	\eta(\boldsymbol{f}_d) & \leq\frac{1}{d^n}\left( d^{n-1}(n+\sqrt{n})\sum_{i=1}^{n}\log||f_{d,i}||_{\sup} \right)\\
	& \leq \frac{1}{d^n}\left( d^{n-1}(n+\sqrt{n})\sum_{i=1}^n\log(d+1) \right)\\
		&\leq K_2\frac{\log d}{d}	
\end{flalign}
for a constant $K_2$ which is independent of $d$. 
On the other hand, by Theorem \ref{th1.4} for $\boldsymbol{f}_d\in Poly_{n,d}\setminus\mathcal{E}_{n,d}$ there exists constants $K_3,K_4,K_5$ and $K_6$ such that
\begin{flalign}
\Delta_{\text{ang}}(Z(\boldsymbol{f}_d)) & \leq K_3\eta(\boldsymbol{f}_d)^{\frac{1}{3}}\log\left( \frac{K_4}{\eta(\boldsymbol{f}_d)} \right)^{\frac{2}{3}(n-1)}\\
& \leq K_5\left(\frac{\log d}{d}\right)^{\frac{1}{3}}\log\left( \frac{d}{\log d} \right)^{\frac{2}{3}(n-1)}  \leq K_6\frac{\log d^{{\frac{2n}{3}}-\frac{1}{3}}}{d^{\frac{1}{3}}}.	
\end{flalign}
  since the function $t^{\frac{1}{3}}\log({\frac{a}{t}})^{\frac{n-1}{3}}$ is increasing for small values of $t>0$. Combining the equations (4.9) and (4.11), we deduce that $\displaystyle\lim_{d\to\infty}\mathbb{E}[\Delta_{\text{ang}}(Z(\boldsymbol{f}_d))]=0$.
 
The proof of the second assertion is analogous and we omit it. 
\end{proof}

\begin{proof}[Proof of Theorem \ref{th:exp}] We adapt the argument in \cite[Theorem 1.8]{DGS} to our setting. Let us denote $\nu_{d}:=\frac{\mathbb{E}[\widetilde{Z}(\boldsymbol{f}_d)]}{d^n}$, where $\mathbb{E}[\widetilde{Z}(\boldsymbol{f}_d)]$ is the expected zero measure and $\nu_{\text{Haar}}$ be the  Haar probability on $(S^1)^n$. We need to show that for each continuous function $\varphi$ with compact support in $\Bbb{C}^n$ we have $\int \varphi d\nu_d\to\int \varphi d\nu_{\text{Haar}}$ as $d\to\infty$. To this end, it is enough to prove the claim for characteristic functions $\varphi_U$ of the open sets  
	\begin{equation}\label{set:open}
		U:=\{(z_1,\ldots,z_n)\in\mathbb{C}^n: r_{1,j}<|z_j|<r_{2,j}\ \text{and}\ \alpha_j<\arg(z_j)<\beta_j\}
	\end{equation}
where $0\leq r_{1,j}<r_{2,j}\leq\infty$, $r_{i,j}\neq 1$ for $i=1,2$ and $-\pi<\alpha_j<\beta_j\leq\pi$. 

First, we consider the case when $U\cap(S^1)^n=\emptyset$. Then one can find an  $0<\varepsilon<1$ such that $U$ is disjoint from the set
\begin{equation}
\{(\xi_1,\ldots,\xi_n)\in\mathbb{C}^n: 1-\varepsilon<|\xi_j|<(1-\varepsilon)^{-1}~ \text{for all $j$}\}.
\end{equation}
Let $\mathcal{E}_{n,d}$ be the exceptional set as in the proof of Theorem \ref{th:main}. If $\boldsymbol{f}_d\in Poly_{n,d}\setminus\mathcal{E}_{n,d}$ then $Z(\boldsymbol{f}_d)$ is discrete and
\begin{equation*}
	\#\{U\cap Z(\boldsymbol{f}_d)\}\leq deg(Z(\boldsymbol{f}_d))\Delta_{\text{rad}}(\boldsymbol{f}_d,\varepsilon)\leq d^n\Delta_{\text{rad}}(\boldsymbol{f}_d,\varepsilon).
\end{equation*}
 On the other hand, if $\boldsymbol{f}_d\in \mathcal{E}_{n,d}$ then by definition $deg(\widetilde{Z}(\boldsymbol{f}_d)|_{U})=0$. Hence, $$\nu_{d}(U)\leq\mathbb{E}[\Delta_{\text{rad}}(\widetilde{Z}(\boldsymbol{f}_d,\varepsilon))]$$ and by Proposition \ref{prop}, $$\lim_{d\to\infty}\int_{Poly_{n,d}}\varphi_U d\nu_d=0=\nu_{\text{Haar}}(U).$$
If $U\cap(S^1)^n\neq\emptyset$ let
\begin{equation}
	\widetilde{U}=\{\boldsymbol{z}: \alpha_j\leq\arg(z_j)\leq\beta_j~~\text{for all $j$ }\}.
\end{equation}
Then we have
\begin{equation*}
 \nu_{d}(U)-\prod_{j=1}^n\frac{\beta_j-\alpha_j}{2\pi} = \left( \nu_{d}(\widetilde{U}) - \prod_{j=1}^n\frac{\beta_j-\alpha_j}{2\pi} \right)-\nu_{d}(\widetilde{U}\setminus U).
\end{equation*}
By Theorem \ref{th:main} we have
\begin{eqnarray}\label{eqnl}
\left| \nu_{d}(\widetilde{U})-\prod_{j=1}^n\frac{\beta_j-\alpha_j}{2\pi}\right|  & \leq& \int_{Poly_{n,d}\setminus\mathcal{E}_{n,d}}\left| \frac{deg(Z(\boldsymbol{f}_d)_{\boldsymbol{\alpha},\boldsymbol{\beta}})}{d^n}-\prod_{j=1}^n\frac{\beta_j-\alpha_j}{2\pi} \right| dProb_d(\boldsymbol{f}_d) \nonumber + \frac{Kn}{d}\\
			& \leq & \int_{Poly_{n,d}\setminus\mathcal{E}_{n,d}}\Delta_{\text{ang}}(Z(\boldsymbol{f}_d)) dProb_d(\boldsymbol{f}_d)+ \frac{Kn}{d}.
\end{eqnarray} 
Note that the set $\widetilde{U}\setminus U$ is a union of a finite number of subsets $U_m$ of the form (\ref{set:open}) such that $U_m\cap(S^1)^n=\emptyset$ for all $m$, we have $\lim_{d\to\infty}\nu_{d}(U_m)=0$ by previous case and hence $\lim_{d\to\infty}\nu_{d}(\overline{U}\setminus U)=0$. Therefore, by Proposition \ref{prop} and (\ref{eqnl}), 
\begin{equation*} 
	\lim_{d\to\infty}\nu_{d}(U)=\lim_{d\to\infty}(\widetilde{U}) =\prod_{j=1}^n\frac{\beta_j-\alpha_j}{2\pi}=\nu_{Haar}(U)
\end{equation*}
which completes the proof. 
\end{proof}



\begin{thebibliography}{99}
	\bibitem{TBI} Bayraktar, T.: Equidistribution of Zeros of Random Holomorphic Sections. \emph{Indiana Univ. Math. J.} \textbf{5} (2016), 1759-1793.
	\bibitem{bay} Bayraktar, T.: Zero distribution of random sparse polynomials. \emph{Michigan Math. J.} \textbf{66} (2017), 389-419.
	\bibitem{B} Bayraktar, T.: Global universality of random zeros. \emph{Hacet. J. Math.} \textbf{48} (2019), 384-398.
        \bibitem{BCHM} Bayraktar, T., Coman, D., Herrmann, H., and Marinescu, G.: A survey on zeros of random holomorphic sections. \emph{Dolomit. Res. Notes Approx.} \textbf{11} (2018), 1-20.
        \bibitem{BBL} Bayraktar, T., Bloom, T. and Levenberg, N.: Random Polynomials in Several Complex Variables, Journal d'Analyse Math., arXiv:2112.00880.
	\bibitem{Ber}  Bernstein, D.N.: The number of roots of a system of equations, \emph{Funktsional. Anal., Prilozhen} \textbf{9} (1975), no.3, 1-4.
	\bibitem{bloom2} Bloom, T.: Random polynomials and (pluri)potential theory, \emph{Ann. Polon. Math.} \textbf{91} (2007), 131-141.
	\bibitem{BD} Bloom, T. and Dauvergne, D.: Asymptotic zero distribution of random orthogonal polynomials. \emph{The Annals of Probability} \textbf{47}(5) 2019, pp.3202-3230.
	\bibitem{BS} Bloom, T. and Shiffman, B.: Zeros of random polynomials on $\mathbb{C}^m$. \emph{Math. Res. Lett.} \textbf{14} (2007), 469-479.
	\bibitem{BL} Bloom, T. and Levenberg, N.: Random Polynomials and Pluripotential Theoretic Extremal Functions. \emph{Potential. Anal.}  \textbf{42} (2015), 311-334. 
	\bibitem{Cox}  Cox, D.A., Little, J. and O'Shea, D.: Using Algebraic Geometry. Second edition, Grad. Texts in Math., 185, Springer, New York, 2005.
	\bibitem{Celik} \c{C}elik, \c{C}.: Equidistribution of Zeros of Random Bernoulli Polynomial Systems, PhD Thesis, (2023), Sabanc{\i} University. 
	\bibitem{DS}  D'Andrea, C. and Sombra, M.: A Poisson Formula for the Sparse Resultant. \emph{Proc. Lond. Math. Soc.} (3) \textbf{110} (2015), no. 4, 932–964.
	\bibitem{DGS} D'Andrea, C. and Galligo, A. and Sombra, M.: Quantitative equidistribution for the solutions of systems of sparse polynomial equations. \emph{Amer. J. of Math.} \textbf{136} (2014), 1543-1579.
	\bibitem{ET} Erd\"{o}s, P. and Tur\'{a}n, P.: On the distribution of roots of polynomials. \emph{Ann. of Math.} \textbf{2} (1950), 105-119.
	\bibitem{GKZ}  Gelfand, I. M., Kapranov, M. M. and  Zelevinsky, A. V.: Discriminants, Resultants, and Multidimensional Determinants, Birkh\"{a}use, 1994.
	\bibitem{Ham}  Hammersley, J.M.: The zeros of random polynomials, Proceedings of the third Berkeley symposium on the mathematical statistics and probability, 1954-1955, vol. II, pp. 89-111.	
	\bibitem{HN} Hughes, C. P.  and  Nikeghbali, A.: The zeros of random polynomials cluster uniformly near the unit circle, \emph{Compos. Math.} \textbf{144} (2008), no. 212, 1541-1555.
	\bibitem{IZ}  Ibragimov, I. and Zeitouni, O. : On roosts of random polynomials. \emph{Trans.Amer. Soc.} \textbf{6} (1997), 2427-2441.
	\bibitem{IZ11} Ibragimov, I. and Zaporozhets, D.: \textit{On Distribution of Random Polynomials in Complex Plane}, Prokhorov and Contemporary Probability Theory, Springer Proc. Math. Stat., 33, Springer, Heidelberg, (2013), 303–323.
	\bibitem{Kac}  Kac, M.: On the average number of real roots of a random algebraic equations. \emph{Bull. Amer. Math. Soc.} \textbf{49} (1943), 314-320.
	\bibitem{KZ} Kozma G. and Zeitoni, I.: On Common Roots of Random Bernoulli Polynomials. \emph{Int. Math. Res. Not.} \textbf{18} (2013), 4334-4347.
	\bibitem{Kush} Kouchnirenko, A. G.: \textit{Poly\`{e}dres de Newton et nombres de Milnor}, Inventiones Mathematicae, (32) 1, (1976 )11-31. 
	\bibitem{LO}
Littlewood J.~E. and Offord, A.~C.: On the number of real roots of a
  random algebraic equation. {III}. \emph{Rec. Math. [Mat. Sbornik] N.S.}
  \textbf{12(54)} (1943), 277--286. 
  
	
	\bibitem{SV}
 Shepp, L.~A. and  Vanderbei, R.~J.: The complex zeros of random
  polynomials, \emph{Trans. Amer. Math. Soc.} \textbf{347} (1995), no.~11,
  4365--4384. 
	\bibitem{S}  Shiffman, B.: Convergence of random zeros on complex manifolds. \emph{Science in China} no.4 Vol 51, (2008), 707-720.
	\bibitem{SZ1} Shiffman, B. and Zelditch, S.: Equilibrium distribution of zeros of random polynomials. \emph{Int. Math. Res. Not.} \textbf{1} (2003), 25-49.
	\bibitem{SZ} Shiffman, B. and Zelditch, S.: Distribution of zeros of random and quantum chaotic sections of positive line bundles. \emph{Comm. Math. Phys.} 200(3):661--683, 1999.
	\bibitem{TV} Tao, T. and Vu, V.: Local Universality of Random Polynomials. \emph{Int. Math. Res. Not. IMRN} (2015), 5053-5139.



\end{thebibliography}
\end{document}